\documentclass[11pt]{amsart}
\usepackage{graphicx}
\usepackage{amssymb}
\usepackage{epstopdf}
\usepackage{amstext,amsmath}
\usepackage{enumerate}
\usepackage{url}
\usepackage{graphicx,color}

\newtheorem{theorem}{Theorem}

\newtheorem{lemma}[theorem]{Lemma}

\newtheorem{remark}[theorem]{Remark}
\newtheorem{example}[theorem]{Example}

\newtheorem{proposition}[theorem]{Proposition}

\def\Z{\mathbb Z}
\def\R{\mathbb R}

\title{Teichm\"uller polynomials of fibered alternating links}
\author{Robert Billet}
\address{Mathematical institute, Florida State University, 600 W College Ave, Tallahassee, FL 32306}
\email{rbillet@math.fsu.edu}
\author{Livio Liechti} 
\address{Department of Mathematics, University of Fribourg, Ch.\ du Mus\'ee 23, 1700 Fribourg, Switzerland}
\email{livio.liechti@unifr.ch}
\thanks{The second author was supported by the Swiss National Science Foundation ($\#159208$)}

\begin{document}
\begin{abstract} We give an algorithm for computing the Teichm\"uller polynomial for a certain class of fibered alternating links
associated to trees. Furthermore, we exhibit a mutant pair of such links distinguished by the Teichm\"uller polynomial.
\end{abstract}
\maketitle

\section{Introduction}   
McMullen introduced the Teichm\"uller polynomial as a geometric analogue of the Alexander polynomial~\cite{McMullen:Poly}. 
It is an invariant of flow equivalence classes of pseudo-Anosov homeomorphisms, 
where two pseudo-Anosov homeomorphisms are \emph{flow-equivalent} if there is a homeomorphism between their mapping tori identifying the 
pseudo-Anosov flows induced by the respective vertical directions of the mapping tori.  
The Teichm\"uller polynomial is used to study all dilatations of elements in a flow equivalence class simultaneously.

Even though the Alexander and Teichm\"uller polynomials are invariants of flow equivalence classes, they
can be computed from a single monodromy. Thought of as invariants of a pseudo-Anosov mapping class $(S,\psi)$, these polynomials give subtle information that goes beyond the
information contained in the characteristic polynomials of the induced action of $\psi$ on the first homology of $S$ or the edge space of a $\psi$-invariant train track $\tau\subset S$, respectively.   
For instance, we will give an example of two mapping classes (defining a mutant pair of fibered alternating links) that are not distinguished 
by the characteristic polynomial of their induced algebraic and geometric actions, but are distinguished by their Alexander and Teichm\"uller polynomials.

The first explicit computations of Teichm\"uller polynomials were done case-by-case~\cite{McMullen:Poly, Hironaka:LT, AD10, KT11}.
More recently, Sun considered a special case of pseudo-Anosov maps arising from Penner's construction~\cite{Sun}, 
and Lanneau and Valdez gave an algorithm computing the Teichm\"uller polynomials of braid monodromies using folding automata for train tracks~\cite{LV16}.

Our goal is to present an algorithm computing the Teichm\"uller polynomial for a large class of fibered alternating links. 
Their fiber surfaces are of genus greater than one and their monodromies are pseudo-Anosov products of multitwists, as in 
constructions developed by Thurston and Penner~\cite{Th, Pe}. 

\subsection{The Teichm\"uller polynomial}
Let~$(S,\psi)$ be a pseudo-Anosov mapping class with mapping torus~$M$. 
Results of Thurston and Fried~\cite{ThNorm, Fried79} show that the flow equivalence class of~$(S,\psi)$ determines a polygonal cone $C \subset H^1(M;\R)$, 
called a \emph{fibered cone}, with the following properties: 
\begin{enumerate}
\item[(i)] for each primitive integral point~$\alpha \in C$, the corresponding map $\alpha_* : \pi_1(M) \rightarrow \Z$ is induced by a fibration~$\rho_\alpha : M \rightarrow S^1$,
\item[(ii)] the fibrations of~$M$ arising in this way are exactly the ones with a monodromy which is flow-equivalent to~$(S,\psi)$. 
\end{enumerate}

The Teichm\"uller polynomial $\Theta_C$ defined by McMullen in \cite{McMullen:Poly} is an invariant of a flow equivalence class, and
is useful for understanding the behavior of the dilatation of its elements.
If $n = \dim(C)$, then $\Theta_C$ is a polynomial in $n$ variables and is well-defined up to change of coordinates and multiplication
by monomials.
More precisely, given a  fibered cone $C$ of a fibered hyperbolic 3-manifold $M$, the corresponding Teichm\"uller polynomial $\Theta_C$
is an element of the group ring $\Z[G]$, where $G = H_1(M;\Z)/\text{torsion}$:
$$
\Theta_C = \sum_{g \in G} a_g x^g \in \Z[G].
$$
The Teichm\"uller polynomial is a geometric analogue of the Alexander polynomial in the following sense: for each primitive integral point $\alpha \in C$, 
the geometric dilatation $\lambda(\psi_\alpha)$ of the monodromy mapping class $\psi_\alpha$ is the largest (real) root of the specialization
$$
\Theta_C^{(\alpha)} (x) = \sum_{g \in G} a_g x^{\alpha(g)}
$$
by a theorem of McMullen~\cite{McMullen:Poly}. 
Similarly, the largest root of the the specialized Alexander polynomial equals the largest eigenvalue of the homological action of $\psi_\alpha$.

\subsection{Alternating-sign Coxeter links}
Our examples for the computation of the Teichm\"uller polynomial are alternating-sign Coxeter links. 
They are constructed as follows. 
To each finite plane tree $\Gamma$, we associate a surface $S$ obtained by thickening vertical and horizontal annuli whose incidence graph is $\Gamma$. 
We define a mapping class $(S,\psi)$ as a composition of a positive multitwist along the horizontal annuli and a negative multitwist along the vertical annuli. 
Such a mapping class $(S,\psi)$ is called an alternating-sign Coxeter mapping class, since its homological action acts, up to a sign, 
as the Coxeter element corresponding to the tree $\Gamma$ with alternating signs~\cite{Hi,HiLi}.
\begin{figure}[h]
\def\svgwidth{160pt}
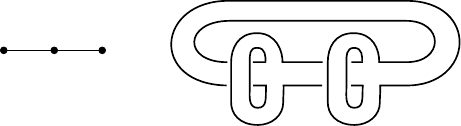
\caption{The tree $\Gamma = A_3$ and its associated surface $S$.}
\label{A3_example}
\end{figure}

Equivalently, we can think of $S$ as the fiber surface for the alternating-sign Coxeter link corresponding to $\Gamma$ and of $\psi$ as its monodromy, see~\cite{HiLi}. 
\begin{figure}[h]
\def\svgwidth{130pt}
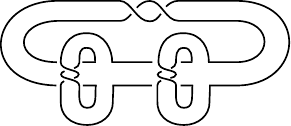
\caption{The fiber surface of the alternating-sign Coxeter link associated to $A_3$ is obtained by adding full twists to the bands of $S$.}
\label{A3_example2}
\end{figure}
In this context, we obtain the ingredients to compute the Teichm\"uller polynomial of the flow-equivalence class of $\psi$ explicitly from the tree $\Gamma$. 
We will see that the number of variables of the Teichm\"uller polynomial equals one plus the rank of the kernel of the adjacency matrix of $\Gamma$.
Therefore, we from now on consider trees $\Gamma$ with a singular adjacency matrix. 

\subsection{Main result}
\label{mainresult}

We describe an algorithm with input a plane tree $\Gamma$ and output the Teichm\"uller polynomial 
$\Theta_{C_\Gamma}$ of the flow-equivalence class of the associated alternating-sign Coxeter mapping class $(S,\psi)$. 
The description of the Teichm\"uller polynomial we use for our computation builds on the cover $\widetilde S\to S$ with the $\psi$-invariant homology $H_{1}(S;\mathbb{Z})^\psi$ as deck group, as described in~\cite{McMullen:Poly}. 
In order to compute the Teichm\"uller polynomial, it is necessary to find matrices for the induced action of a lift $\widetilde\psi$ of $\psi$ on the lift $\widetilde\tau$ of a $\psi$-invariant train track $\tau\subset S$. 
In our algorithm, we reduce this problem to one which can be solved explicitly using only the structure of the plane tree $\Gamma$.
A sample of this reduction is given below.   
\newline

\noindent
\textbf{Illustration of the algorithm.}
Let $\Gamma$ be a plane tree. In the following, we will define matrices $U,V,W$ and $T$, obtained explicitly using only the structure of $\Gamma$. 
These matrices determine the specialization to a two-dimensional slice of the Teichm\"uller polynomial $\Theta_{C_\Gamma}$ 
of the fibered cone $C_\Gamma$ containing the monodromy of the alternating-sign Coxeter link corresponding to $\Gamma$. 

\begin{theorem}
\label{sample_theorem}
The fibered cone $C_\Gamma$ of the alternating-sign Coxeter link associated to $\Gamma$ has a two-dimensional slice with coordinates $(x,u)$
 to which the Teichm\"uller polynomial $\Theta_{C_\Gamma}$ specializes as
\begin{displaymath}
\Theta_{C_\Gamma}(x,u) = \frac{\det(uI-UWVT)}{\det(uI-WT)^{\frac{1}{2}}}.
\end{displaymath}
\end{theorem}  

We now show how to define the matrices $U,V,W$ and $T$ from the plane tree $\Gamma$. 
Let $\mathcal{A}$ be an arrangement of vertical and horizontal segments $a_i$ and $b_j$ in the plane with incidence graph $\Gamma$. 
Identifying top and bottom of every vertical segment $a_i$ and identifying left and right of every horizontal segment $b_j$ yields 
a graph $\tau$ with edges $e_0,\dots e_n$, see Figure~\ref{A3_example3} for the example $\Gamma=A_3$.
We define the following maps. 

\begin{figure}
\def\svgwidth{290pt}
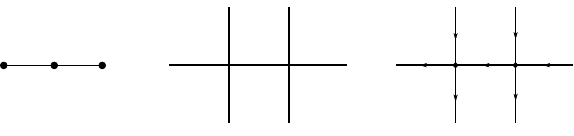
\caption{The tree $\Gamma = A_3$, the arrangement $\mathcal{A}$ with incidence graph $A_3$ (the vertex of degree 2 in $\Gamma$ corresponds to the segment $b_1$) and the graph $\tau$.}
\label{A3_example3}
\end{figure}

\begin{enumerate}
\item Let $s,t:E(\tau)\to V(\tau)$ be the maps associating to every edge of $\tau$ its starting and terminal vertex, respectively, when oriented downwards and to the left. 
\item Let $l:V(\Gamma)\to V(\tau)$ associate to every vertex of $\Gamma$ the uppermost or rightmost vertex on the corresponding vertical or horizontal segment, respectively.  
\item Let $d: E(\tau)\to V(\Gamma)$ be the map associating to an edge $e$ of $\tau$ the vertex of $\Gamma$ in whose corresponding segment $e$ is contained.
\item Let $c: E(\tau)\times V(\tau)\times V(\tau) \to \{0,1\}$ be the map assigning $1$ to $(e,v_i,v_j)$ if $e$ is contained in the convex hull of $v_i$ and $v_j$ and $0$ otherwise.
\item Let $p: E(\tau) \to \{0,1\}$ be the map assigning $1$ to vertical edges and $0$ to horizontal ones.
\item Let $g: E(\tau) \times V(\tau) \to \{0,1\}$ be the map assigning $1$ to $(e,v)$ if $e$ and $v$ lie on a straight line and $0$ otherwise. 
\item Let $v_0 \in V(\tau)$ be fixed. For every vertex $v_i \in V(\tau)$, there is a unique simple oriented edge path $\gamma_{v_i}$ in $\tau$ connecting $v_0$ to $v_i$ without using any of the edges with top-bottom and left-right identifications. 
Let $\text{or}(e)$ of an oriented edge $e$ be $+1$ if the orientation matches the orientation of the edge in $\tau$ as in (1) and $-1$ otherwise.
\item Let $a: E(\tau)\to V(\Gamma)$ be the map assigning to a horizontal edge $e$ the vertex of $\Gamma$ corresponding to the segment perpendicular to $e$ at its terminal vertex, 
and to a vertical edge $e$ the vertex of $\Gamma$ corresponding to the segment perpendicular to $e$ at its starting vertex. 
\end{enumerate}

\noindent
If $k = (k_1,\dots,k_{\vert V\vert}) \in \mathbb{Z}^{\vert V(\Gamma)\vert}$ is a given primitive element of the kernel of the adjacency matrix of $\Gamma$,
define matrices $U = I + (m_{ij})$ and $V= I + (n_{ij})$ of size $|E(\tau)|\times |E(\tau)|$ by
\begin{align*} m_{ij} &= p(e_i)( 1-p(e_j)) g(e_i,t(e_j))x^{c(e_i,l(d(e_i)),t(e_j))k_{d(e_i)}},\\
n_{ij} &= p(e_j)( 1-p(e_i)) g(e_i,s(e_j))x^{c(e_i,l(d(e_i)),s(e_j))k_{d(e_i)}}.
\end{align*} 
For example, $m_{ij}\ne0$ if $e_i$ is vertical and part of the segment perpendicular to the terminal vertex of the horizontal edge $e_j$. 
In this case, $m_{ij}=x^{k_{d(e_i)}}$ if $e_i$ lies between the terminal vertex of $e_j$ and the uppermost vertex of the segment perpendicular to the terminal vertex of $e_j$, and 
$m_{ij}=1$ otherwise.

\noindent
Define two diagonal matrices $W = (w_{ij})$ and $T=(t_{ij})$ of size $|E(\tau)|\times |E(\tau)|$ by
\begin{align*} w_{ii} &= \prod_{e\in\gamma_{s(e_i)}}x^{( 1-p(e))\text{or}(e)k_{a(e)}},\\
t_{ii} &= \prod_{e\in\gamma_{s(e_i)}} x^{p(e)\text{or}(e)k_{a(e)}},
\end{align*}
where the product is taken over the $e \in E(\tau)$ that appear in the path~$\gamma_{s(e_i)}$. 
The matrices~$U$,~$V$,~$W$ and~$T$ we just defined are the ones to be used in the formula of Theorem~\ref{sample_theorem}.
Note that since every vertex $v\in V(\tau)$ occurs twice as a starting point of an edge $e\in E(\tau)$, the denominator in Theorem~\ref{sample_theorem} is also a polynomial.

\begin{remark}\emph{
 Alternating-sign Coxeter mapping classes~$(S,\psi)$ preserve the orientation of the invariant train track. 
 Equivalently, they preserve the orientation of the corresponding invariant foliations.
 Thus, the induced foliations on the mapping torus are orientable, and every other element in the flow equivalence class, given by a cross section to the suspension flow of~$(S,\psi)$, is also orientable.
 In fact, one can show that the Alexander polynomial and the Teichm\"uller polynomial agree for alternating-sign Coxeter links corresponding to trees.  
 This gives an alternative way of computing the Teichm\"uller polynomial for these cases.
 However, our focus lies on the perspective corresponding to the Teichm\"uller polynomial.
}\end{remark}

\subsection{Application to mutant links}

Let $B\subset \mathbb{S}^3$ be a ball whose boundary sphere $\partial B$ meets a link $L\subset \mathbb{S}^3$ transversally and symmetrically in four points. 
The operation of exchanging the interior of $B$ with its image under a sequence of reflections is a \emph{mutation}.  
Two links are \emph{mutant} if they are connected by a sequence of mutations.
Mutants are hard to distinguish. For example, by a theorem of Wehrli~\cite{We}, an invariant $f$ that is able to distinguish a pair of mutant links cannot satisfy a skein relation of the form
$$\alpha f(L_+) + \beta f(L_-) + \gamma f(L_0) = 0,$$
where $\alpha,\beta\in R^\ast$ and $\gamma\in R$ are fixed elements of an arbitrary ring $R$. 
For instance, the single variable Jones and Alexander polynomials are not able to distinguish mutant links.
We give an example of mutant alternating-sign Coxeter links distinguished by their Teichm\"uller polynomials. 
In particular, it follows that there is no skein relation of the above form for the Teichm\"uller polynomial of fibered links.
\begin{proposition}
\label{mutant_proposition}
There exists a pair of fibered alternating mutant links distinguished by their Teichm\"uller polynomials.
\end{proposition} 
In our examples of alternating-sign Coxeter links, the Teichm\"uller polynomial coincides with the multivariable Alexander polynomial. 
In particular, Proposition~\ref{mutant_proposition} also applies to the multivariable Alexander polynomial.
This is in contrast with the existence of
more involved skein relations for the multivariable Alexander polynomial~\cite{Ji} and
the fact that the multivariable Alexander polynomial is invariant under many mutations~\cite{Zi}. \\

\noindent
{\bf Acknowledgements.} 
We warmly thank Eriko Hironaka for initiating our collaboration on this project, and for many inspiring discussions and valuable suggestions.
We also thank Sebastian Baader, Bal\'azs Strenner, Stephan Wehrli and the anonymous referee for helpful comments and references.

\subsection{Organization}
In Section~\ref{background_section}, we give the necessary background on the Teichm\"uller polynomial. 
In Section~\ref{constructions-sec}, we recall the construction of alternating-sign Coxeter links, their fiber surfaces and their monodromies.
In Section~\ref{algorithm_section}, we describe a general algorithm computing the Teichm\"uller polynomial of a fibered cone containing the alternating-sign Coxeter link associated to a tree. 
In particular, we will prove Theorem~\ref{sample_theorem}.
Finally, in Section~\ref{example_section}, we describe the Teichm\"uller polynomial of the alternating-sign Coxeter links associated to the trees $A_n$, 
give a detailed computation for $A_5$ and prove Proposition~\ref{mutant_proposition}.

\section{Dilatation, fibered faces, and the Teichm\"uller polynomial}
\label{background_section}
In this section, we recall the necessary notions concerning the Teichm\"uller polynomial and describe the formula we will use to compute it.

\subsection{Dilatation}
Let $M$ be a 3-manifold that fibers over the circle.   
Then $M$ has the structure of a mapping torus
$M = S \times [0,1]/(x,1) \sim (\psi(x),0)$, where $S$ is a 
surface and  $\psi : S \rightarrow S$ is a homeomorphism defined up to isotopy and conjugation.  
If $M$ is hyperbolic, then by a fundamental result of Thurston,
the {\it monodromy} $(S,\psi)$  is {\it pseudo-Anosov}~\cite{Th2} and has
an associated {\it dilatation} (or {\it stretch factor}) $\lambda = \lambda(\psi) > 1$,
which can be characterized in any of several equivalent ways:
\begin{enumerate}[(i)]
\item ({\it pseudo-Anosov property}) there is a pair of $\psi$-invariant transverse measured singular foliations $({\mathcal F}^{\pm}, \nu^{\pm})$ so that 
$\psi_*(\nu^\pm) = \lambda \nu^\pm$;
\item ({\it exponential growth rate of lengths of curves}) for any Riemannian
metric $\omega$ on $S$, and any essential simple closed curve $\gamma$ on $S$, we have
$$
\lambda = \lim_{n \rightarrow \infty} \ell(\phi^n(\gamma))^{\frac{1}{n}};
$$
\item ({\it entropy}) $\log(\lambda)$ is  the minimal topological entropy of any representative of the isotopy class of homeomorphisms
defined by $\phi$; and
\item ({\it geodesics on moduli space}) the orbits of $\phi$ in Teichm\"uller space determine a geodesic on the moduli space of $S$ whose Teichm\"uller length is $\log(\lambda)$.
\end{enumerate}
See~\cite{CB88, Primer, FLP, Th} for more details and background.

\subsection{Fibered faces}
The mapping classes we are considering are pseudo-Anosov, so each mapping torus $M$ is hyperbolic. 
Thus, the Thurston norm $\| \psi \|_T$ defined on $H^1(M;\mathbb{Z})$ is non-degenerate and has convex norm ball $B_{x}$~\cite{ThNorm}. 
Let $F$ be a top-dimensional face of $B_{x}$ and let~$CF=\mathbb{R}^+\cdot F$ be the cone over $F$. 
We say $F$ is a {\it fibered face} of the Thurston norm ball if every integral element~$\alpha$ in $CF$ corresponds to a fibration~$\rho_\alpha:M \rightarrow S^1$.
By a theorem of Fried~\cite{Fried82}, proximity in the projectivization~$F$ of the cone~$CF$ implies closeness
for the normalized dilatation~$L(S_\alpha,\psi_\alpha) = \lambda(\psi_\alpha)^{|\chi(S_\alpha)|}$, where~$\psi_\alpha:S_\alpha\to S_\alpha$ is the monodromy of~$\rho_\alpha$.  
In fact,~$L$ extends to a convex, real analytic
function on~$F$ that goes to infinity toward the boundary of~$F$.
\subsection{The Teichm\"uller polynomial}
Let~$G=H_1(M;\mathbb{Z})/\text{torsion}$. For each fibered face~$F$ of the Thurston norm ball on $H^1(M;\mathbb{Z})$, McMullen~\cite{McMullen:Poly} defined an element 
$$
\Theta_{CF} = \sum_{g \in G} a_g x^g \in \Z[G],
$$
called the \emph{Teichm\"uller polynomial} of the fibered face $F$ (or the fibered cone $CF$),
which comes packaged with the dilatation of each primitive integral class $\alpha_{\psi} \in CF$: 
the dilatation $\lambda(\psi_\alpha)$ is the largest (real) root of the specialization
$$
\Theta_{CF}^{(\alpha)} (x) = \sum_{g \in G} a_g x^{\alpha(g)}.
$$

McMullen also gives a formula for the Teichm\"uller polynomial $\Theta_{C}$, see~\cite{McMullen:Poly}. 
We quickly review the statement. Fix a fiber $\iota:S \rightarrow M$ and let $\psi: S \rightarrow S$ be the associated pseudo-Anosov monodromy with $\psi$-invariant cohomology $H^1(S;\mathbb{Z})^\psi$. 
Let $H=\text{Hom}(H^1(S;\mathbb{Z})^\psi,\mathbb{Z})$ and let $\widetilde{S}$ be the cover defined by~$\iota_\ast:\pi_{1}(S) \rightarrow H$. 
We think of $\widetilde{S}$ as a component of the preimage of the chosen fiber $S$ under the covering map $\pi : \widetilde{M} \rightarrow M$ where $\widetilde{M}$ is the maximal abelian cover of $M$. 
From this point of view, $H$ is the subgroup of Deck$(\pi)$ that fixes $\widetilde{S}$. We can think of $\widetilde{M}$ as 
\begin{displaymath}
\bigsqcup\limits_{i \in \mathbb{Z}} \left(\widetilde{S}_{i} \times [0,1]\right),
\end{displaymath}
where $\widetilde{S}_{i}\cong \widetilde{S}$ and $(s,1)$ in $\widetilde{S}_{i}$ is identified with $(\widetilde{\psi}(s),0)$ in $\widetilde{S}_{i+1}$. 

Let $\tau$ be a connected graph with an embedding $j:\tau \rightarrow S$ so that the edges are tangent at the vertices. 
We can partition the edges at each vertex into two sets, thought of as incoming and outgoing edges. 
The region created by two consecutive edges (in the cyclic order determined by $j$) of the same type is called a cusp. 
Then $\tau$ is a \emph{train track} if no vertex has degree $1$ or $2$ and the connected components of $S\setminus j(\tau)$ are either polygons with at least one cusp or annuli with one boundary contained in $\partial S$ and the other boundary with at least one cusp. 
Associated to a pseudo-Anosov mapping class $\psi$ there is a $\psi$-invariant train track $\tau$. 
In particular, there is an action on the space of weights on the edges of $\tau$ induced by $\psi$, sending an edge $e$ to the edges (counted with multiplicity) in the edge path given by $\psi(e)$.
Similarly, there is an induced action on the space of weights on the vertices of $\tau$.

Let and $E$ and $V$ denote the edges and vertices of a $\psi$-invariant train track $\tau \subset S$.  
Let $(t_{1},...,t_{b},u)$ be a multiplicative basis for 
$$
H_{1}(M;\mathbb{Z})/\text{torsion}\cong \text{Hom}(H^{1}(S;\mathbb{Z})^{\psi},\mathbb{Z})\oplus \mathbb{Z}.
$$
The action of $H$ on $\widetilde{S}$ restricts to an action on the lifts of $E$ and $V$. 
Thus, we can think of a lift $\widetilde{\psi}$ of $\psi$ as acting on the free $\mathbb{Z}[H]$-modules generated by the lifts of $E$ and $V$. 
Let the action of $\widetilde{\psi}$ on the edge and vertex space of the lifted train track $\widetilde\tau$ (thought of as $\mathbb{Z}[H]$-modules) be given by the matrices $P_{E}(t)$ and $P_{V}(t)$. 
Then McMullen's determinant formula states that the Teichm\"uller polynomial for the fibered cone $C$ containing $[S]$ is given by 
\begin{displaymath}
\Theta_C(t,u) = \frac{\det(uI-P_{E}(t))}{\det(uI-P_{V}(t))},
\end{displaymath}
independently of any choices~\cite{McMullen:Poly}.


\section{Alternating-sign Coxeter links}
\label{constructions-sec}
In this section, we associate fibered alternating links $L$ to trees $\Gamma$ and we describe the monodromy $\psi$ of their fiber surface $S$. 
More precisely, we construct an invariant train track $\tau\subset S$ for $\psi$ and describe the dual $\text{Hom}(H^{1}(S;\mathbb{Z})^{\psi},\mathbb{Z})$ of the $\psi$-invariant cohomology $H^{1}(S;\mathbb{Z})^{\psi}$ in terms of the adjacency matrix of $\Gamma$.
While we restrict our attention to trees, 
much of what we are about to do generalizes to a more general class of bipartite Coxeter graphs. For more details on this construction and Coxeter links and mapping classes, see~\cite{Hi, HiLi}.

Let $\Gamma$ be a finite \emph{plane} tree, that is, a finite tree together with an embedding in the plane. 
For every vertex $v$ of $\Gamma$, this embedding defines a circular ordering of the edges touching $v$.
It is possible to draw horizontal and vertical segments in the plane such that their incidence graph equals $\Gamma$ and respects the circular ordering of edges at every vertex. 
We thicken the union of all segments and attach bands with one full twist to the ends of each segment, as in Figure~\ref{construction_figure}.
\begin{figure}[h]
\def\svgwidth{240pt}
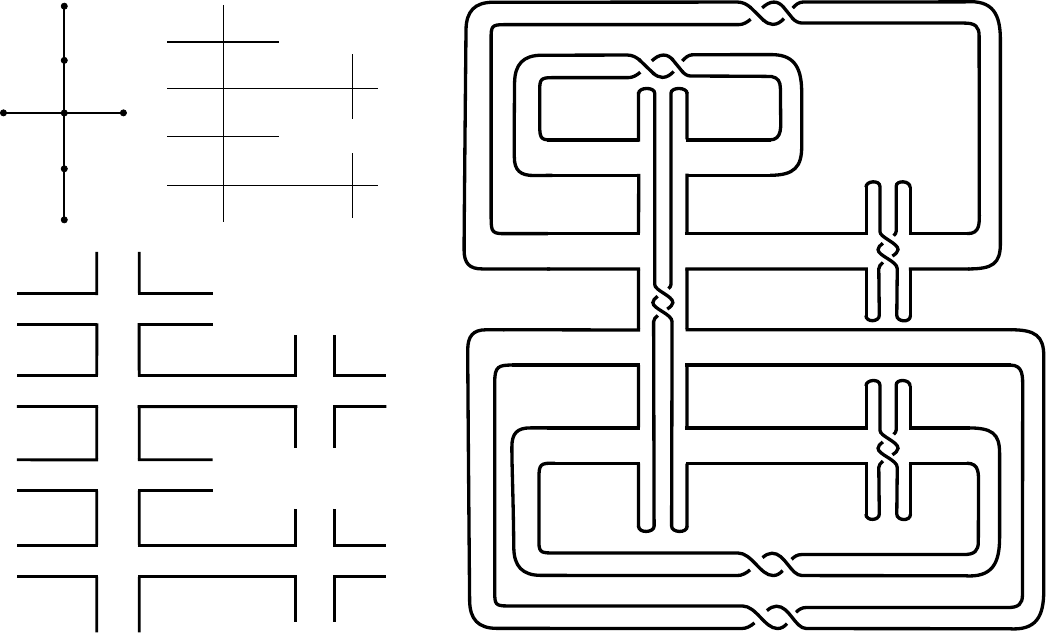
\caption{A plane tree $\Gamma$, horizontal and vertical segments in the plane with $\Gamma$ as incidence graph, thickened segments and the fiber surface $S$ obtained by adding twisted bands.}
\label{construction_figure}
\end{figure}
The resulting surface $S$ is an iterated Murasugi sum of Hopf bands, hence $S$ is a fiber surface for the boundary link $L=\partial S$ and 
the monodromy is given by a product of Dehn twists along the core curves of the Hopf bands~\cite{Ga}.
In the classical case, each Hopf band is positive. Equivalently, each Dehn twist is positive.
In the mixed-sign case~\cite{Hi}, the tree $\Gamma$ comes equipped with a function $\mathfrak{s}: V(\Gamma) \rightarrow \{-1,1\}$. 
Each vertex marked with a negative sign then corresponds to a negative Hopf band and hence a negative Dehn twist. 
In what follows, we consider signs $\mathfrak{s}$ respecting the bipartition of the tree $\Gamma$. 
An \textit{alternating-sign Coxeter link} is a link arising from a tree $\Gamma$ with a bipartition $V(\Gamma)=V_{1}\sqcup V_{2}$, 
together with a labeling $\mathfrak{s}$ so that $\mathfrak{s}(V_{1})=\pm1$ and $\mathfrak{s}(V_{2})=\mp1$.
We have already seen that alternating-sign Coxeter links are fibered. 
Furthermore, one can show they are alternating~\cite{HiLi}. 
We regard their monodromy $\psi: S\to S$ as the product of a positive Dehn twist along all core curves of the positive Hopf bands and a negative Dehn twist along all core curves of the negative Hopf bands.
Indeed, for any tree $\Gamma$, the constructed mapping class does, up to conjugation, not depend on the order of Dehn twists~\cite{Steinberg}.

\begin{remark}
\label{mutation_remark}
 \emph{It is possible for a given abstract tree $\Gamma$ to have several non-equivalent planar embeddings, and for the corresponding diagrams to give mutant but distinct links. 
Figure~\ref{mutation_figure} shows a red sphere cutting a link in four points. Rotating by $\pi$ exchanges the two vertical bands shown.
Using isotopies and such operations, one can arbitrarily change the circular ordering of edges at any vertex of $\Gamma$. 
This shows that the alternating-sign Coxeter links associated to any two planar embeddings of a given abstract tree $\Gamma$ are mutants. 
While the single variable Alexander and Teichm\"uller polynomial can not distinguish these links, 
in Section~\ref{example_section} we provide an example where two distinct embeddings of an abstract tree result in a pair of links that have different Teichm\"uller polynomials.}
\end{remark}

\begin{figure}[h]
\def\svgwidth{100pt}
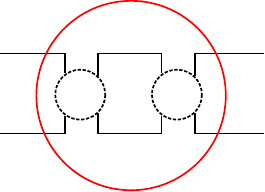
\caption{}
\label{mutation_figure}
\end{figure}

By a theorem of Penner~\cite{Pe}, the mapping class $(S,\psi)$ constructed above is pseudo-Anosov as soon as $\Gamma$ has at least two vertices. 
Furthermore, an orientable $\psi$-invariant train track $\tau\subset S$ is obtained by smoothing the union of core curves of the Hopf bands at their intersection points, as in Figure~\ref{smoothing_figure}. 
\begin{figure}[h]
\def\svgwidth{200pt}
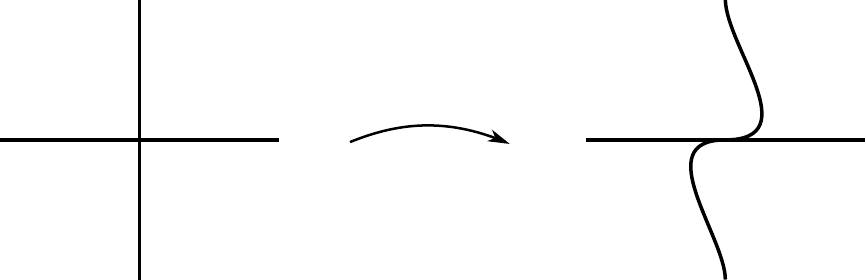
\caption{The local smoothing operation used to build an invariant train track out of the core curves of the Hopf bands.}
\label{smoothing_figure}
\end{figure}
An orientation of $\tau$ is then given by orienting all horizontal edges to the left and all vertical edges downwards.

\begin{lemma}
If the tree $\Gamma$ has $\vert E\vert$ edges, then the $\psi$-invariant train track $\tau\subset S$ has $2\vert E\vert$ edges. 
\end{lemma}
\begin{proof}
We construct $\tau$ from a collection of horizontal and vertical segments, where these collections of segments correspond to the bipartition of $\Gamma$ and intersect according to their adjacency in $\Gamma$. Then, the segment corresponding to some vertex $v$ is subdivided into $\text{deg}(v)$ edges. Thus, the train track $\tau$ has $\sum\limits_{v} \text{deg}(v)=2\vert E\vert$ edges. 
\end{proof}

In order to construct the covers necessary to compute the Teichm\"uller polynomial, we need to know the dual $\text{Hom}(H^{1}(S;\mathbb{Z})^{\psi},\mathbb{Z})$ of the $\psi$-invariant cohomology~$H^{1}(S;\mathbb{Z})^{\psi}$ of our constructed surface $S$.
This is the content of the following lemma.
From the construction, we have $H_{1}(S;\mathbb{Z})\cong\mathbb{Z}^{\vert V\vert}$, where $V$ is the set of vertices of the original tree $\Gamma$.

\begin{lemma}
\label{invariant_lemma}
Let $(S,\psi)$ be a mapping class constructed from a plane tree $\Gamma$ as above. Let $A$ be the adjacency matrix of $\Gamma$. If we regard $A$ as a linear map $H_{1}(S;\mathbb{Z}) \rightarrow H_{1}(S;\mathbb{Z})$, then $H_{1}(S;\mathbb{Z})^{\psi}=\ker(A)$.  
\end{lemma}

\begin{proof}
By labeling the vertices one set of the bipartition at a time, we can arrange $A$ to have the form
\begin{displaymath}
A=\begin{pmatrix}
0 & X \\
X^\top & 0
\end{pmatrix}.
\end{displaymath}
Then the action of $\psi$ on $H_{1}(S;\mathbb{Z})$ can be expressed as the product of the two multitwists
\begin{displaymath}
T_a=\begin{pmatrix}
I & X \\
0 & I
\end{pmatrix},\ T_b=\begin{pmatrix}
I & 0 \\
X^\top & I
\end{pmatrix}.
\end{displaymath}
Let $v\in H_{1}(S;\mathbb{Z})$ be expressed according to the bipartition as $v=\begin{pmatrix}
v_{1} \\
v_{2}
\end{pmatrix}$.
Then $\psi_{*}(v)=v$ holds if and only if
$$
v_{1} + XX^\top v_{1} +Xv_{2}=v_{1},
$$
$$
X^\top v_{1} + v_{2} = v_{2}.
$$
But this is equivalent to 
\begin{displaymath}
\begin{pmatrix}
0 & X \\
X^\top & 0
\end{pmatrix} 
\begin{pmatrix}
v_{1} \\
v_{2}
\end{pmatrix}
= 
\begin{pmatrix}
Xv_{2} \\
X^\top v_{1}
\end{pmatrix} = 0,
\end{displaymath}
which is what we wanted to show.
\end{proof}
In order to obtain the dual $\text{Hom}(H^{1}(S;\mathbb{Z})^{\psi},\mathbb{Z})$ of the $\psi$-invariant cohomology $H^{1}(S;\mathbb{Z})^{\psi}$, 
we remark that the matrix product~$T_aT_b$ describing the action~$\psi_\ast$ of~$\psi$ induced on the first homology is symmetric. In particular, the dual $\text{Hom}(H^{1}(S;\mathbb{Z})^{\psi},\mathbb{Z})$ of the $\psi$-invariant cohomology $H^{1}(S;\mathbb{Z})^{\psi}$
is simply the $\psi$-invariant homology $H_{1}(S;\mathbb{Z})^{\psi}$.

\section{Algorithm}
\label{algorithm_section}

The aim of this section is to present a general way to compute the Teichm\"uller polynomial of the fibred cone $C_\Gamma$ corresponding to the alternating-sign Coxeter link for any plane tree $\Gamma$. 
In order to apply our algorithm in practice and obtain the Teichm\"uller polynomial, additional choices have to be made. 
For explicit calculations, see Section~\ref{example_section}.

Let $\Gamma$ be a plane tree and let $S$ and $\tau$ be the surface and the train track, respectively, constructed from $\Gamma$ as in Section~\ref{constructions-sec}. 
We denote the generators for the homology $H_{1}(M;\mathbb{Z})$ and the $\psi$-invariant homology of $S$ according to the isomorphism 
\begin{displaymath}
H_{1}(M;\mathbb{Z})/\text{torsion} \cong H_{1}(S;\mathbb{Z})^\psi \oplus \mathbb{Z} \cong <x_{0},...,x_{n-1}> \oplus <x_{n}>. 
\end{displaymath}

\noindent
The algorithm is divided into three main steps:

\begin{enumerate} 
\item Construct $(S,\psi)$ together with a $\psi$-invariant train track $\tau \subset S$.
\item Compute $H_{1}(S;\mathbb{Z})^\psi$, construct the cover $\widetilde{S} \to S$ corresponding to $\pi_1(S)\to H_{1}(S;\mathbb{Z})^\psi$ 
and lift the train track $\tau$ to obtain $\widetilde{\tau}$.
\item Decompose the action of $\widetilde{\psi}$ on the edges and the vertices of $\widetilde{\tau}$ as a product of two multitwists on a single fundamental domain of $\widetilde\tau$.
\end{enumerate}

\textit{Step 1.} We quickly recall the construction from Section~\ref{constructions-sec}. 
Given a plane tree $\Gamma$, the surface $S$ is constructed by suitably gluing horizontal and vertical bands, whose core curves have $\Gamma$ as incidence graph. 
The corresponding mapping class $(S,\psi)$ is defined to be a product of a positive and a negative multitwist along the core curves of the horizontal and vertical bands, respectively. 
The $\psi$-invariant train track $\tau \subset S$ is then obtained by smoothing the crossings of the core curves of the bands.
The induced action of $\psi$ on the train track $\tau$ can be decomposed as a product $T_aT_b$ of the induced actions $T_a$ and $T_b$ of the positive and negative multitwist, respectively. 
To make this induced action explicit, let the bipartition of $\Gamma$ be given by $V(\Gamma)=V_{1} \sqcup V_{2}$, 
where the vertices of $V_1$ and $V_2$ correspond to vertical and horizontal bands, respectively.   
Let $\mathfrak{s}: V_{1} \sqcup V_{2} \rightarrow \{1,-1\}$ so that $\mathfrak{s}(V_{1})=-1$ and $\mathfrak{s}(V_{2})=1$ be our labeling. 
We choose a basis of the edge space of $\tau$ in the following way.
First we take all the edges of the train track $\tau$ that correspond to $V_{1}$, that is, to bands along whose core curves we twist negatively.
Then we take all the edges that correspond to $V_2$. 
The induced actions $T_a$ and $T_b$ of the negative and positive multitwist, respectively, now have the form  
\begin{displaymath}
T_{a}=\begin{pmatrix}
I & X \\
0 & I
\end{pmatrix},
\hspace{30pt} 
T_{b}=\begin{pmatrix}
I & 0 \\
Y & I
\end{pmatrix}.
\end{displaymath}
At this point, it is possible to compute the dilatation of the mapping class $(S,\psi)$ corresponding to the fixed fiber $S$ as the largest root of
\begin{displaymath}
\det(xI-T_{a}T_{b}).
\end{displaymath}

\textit{Step 2.} We first explain how to construct $\widetilde{S}$, the cover of $S$ corresponding to $\pi_{1}(S) \rightarrow \text{Hom}(H^{1}(S;\mathbb{Z})^{\psi},\mathbb{Z})$. 
Let $A$ be the adjacency matrix of $\Gamma$.
In the basis induced by the vertices of $\Gamma$, 
the $\psi$-invariant homology $H_{1}(S;\mathbb{Z})^\psi$ is the kernel $\text{ker}(A)$ by Lemma~\ref{invariant_lemma}.
As a submodule of a finitely generated free $\mathbb{Z}$-module, $\text{ker}(A)$ is finitely generated and admits a basis. 
Let $B$ be a matrix whose columns form a basis for $\text{ker}(A)$. 
By the correspondence of $\text{ker}(A)$ and $H_{1}(S;\mathbb{Z})^{\psi}$, each column of $B$ represents an element of a basis of $H_{1}(S;\mathbb{Z})^{\psi}$. 
On the other hand, each row of $B$ represents a vertex $v_i$ of $\Gamma$ and hence an annulus~$A_i$ in $S$. 
A nonzero entry $B_{ij}\ne0$ signifies that a loop winding around the core curve of the annulus $A_{i}$ should lift to a path in $\widetilde{S}$ corresponding to $x^{B_{ij}}_{j}$ as an element of the deck transformation group $H_{1}(S;\mathbb{Z})^{\psi}$. With this correspondence in mind, we define a function $f$ on the vertices of $\Gamma$ as follows:
\begin{displaymath}
v_{i} \mapsto \prod\limits_{j}{x_{j}^{B_{ij}}}.
\end{displaymath}
Then the cover $\widetilde{S}$ is constructed by cutting copies of $S$ along cohomology classes (thought of as arcs dual to bands) fixed by $\psi$ and attaching one end of each cut band $v$ with the other end of the equivalent band moved by what will be the deck transformation $f(v)$. Having constructed the cover $\widetilde{S} \to S$, we also obtain a natural lift $\widetilde{\tau}$ of $\tau$. 

\textit{Step 3.} Choose a lift $\widetilde\psi$ of $\psi$. We can compute the action of $\widetilde{\psi}$ on the edge space of $\widetilde{\tau}$ by looking at a single fundamental domain of $\widetilde{S}$. 
For this, recall from Section~\ref{background_section} that we regard the edge space of $\widetilde\tau$ as a $\mathbb{Z}[H]$-module.
Let $\widetilde{T}_{a}$ and $\widetilde{T}_{b}$ be the actions of the lifted multitwists on the edge space of $\widetilde{\tau}$. 
Then the matrix product $\widetilde{T}_{a} \widetilde{T}_{b}$ equals the action of $\widetilde\psi$ on the edge space of $\widetilde\tau$.
We now explain how both $\widetilde{T}_{a}$ and $\widetilde{T}_{b}$ can be written as a product of two matrices which in turn can be obtained directly from our given data.

Let $U$ and $V$ be matrices that differ from $T_{a}$ and $T_{b}$ only in that they take into account that when a curve winds around the core of the band corresponding to $v_{i}$, 
it moves to the $f(v_{i})$-level of $\widetilde{S}$. 
More precisely, the image after such a winding of an edge of $\widetilde\tau$ is multiplied by the coefficient $f(v_{i})$. 
In particular, the entries of the matrices $U$ and $V$ are polynomials in the variables $x_0, \dots , x_{n-1}$, and specializing every $x_i$ to $1$ yields the original matrices ${T}_{a}$ and ${T}_{b}$.
It is important to note that the matrices $U$ and $V$ might still differ from the matrices $\widetilde{T}_{a}$ and $\widetilde{T}_{b}$, respectively.
Indeed, it is possible that the image of an edge of $\widetilde\tau$ under a lifted multitwist starts at a different level of $\widetilde{S}$. 
Equivalently, a lifted multitwist does not necessarily fix the vertices of $\widetilde\tau$ pointwise.
However, this can be accounted for by multiplication of the columns of $U$ and $V$ corresponding to the affected edges by coefficients.  
This can be done by right multiplication with diagonal matrices $W$ and $T$, respectively, yielding $\widetilde{T}_{a} = UW$ and $\widetilde{T}_{b} =VT$.
Note that the correction process which determines the matrices $W$ and $T$ depends on the choice of a lift $\widetilde\psi$.
Summarizing, we can express $P_{E}(x_{0},...,x_{n-1})$ as a product $UWVT$. 

Since the action of a lifted multitwist on the vertices of the lifted train track $\widetilde\tau$ shifts levels as does the action on the edges starting at this vertex, we 
see that $P_{V}(x_{0},...,x_{n-1})$ is the diagonal matrix given by the upper (or lower) half diagonal block of the matrix product $WT$. Finally, according to McMullen's determinant formula, we have
\begin{displaymath}
\Theta_{C_\Gamma}(x_{0},...,x_{n})=\frac{\det(x_{n}I-P_{E}(x_{0},...,x_{n-1}))}{\det(x_{n}I-P_{V}(x_{0},...,x_{n-1}))}.
\end{displaymath}

\begin{proof}[Proof of Theorem~\ref{sample_theorem}]
The matrices $U,V,W$ and $T$ described in Theorem~\ref{sample_theorem} are obtained exactly as described in the algorithm above.
We would like to highlight two things about this description. 
Firstly, some parts of the description determine choices in the algorithm:
the map $a$ defined in (8) in Section~\ref{mainresult} determines the choice of a representative of the multitwists $T_a$ and $T_b$ twisting to the right and below the core curves in the thickened arrangement of vertical and horizontal lines,
the choice of a vertex $v_0\in V(\tau)$ and a path $\gamma_{v_i}$ connecting $v_0$ to any other vertex $v_i$, as in (7), determines a lift of $\widetilde\psi$, 
and the map $c$ defined in (4) determines the choice of a fundamental domain of the lifted train track.
Secondly, the matrices described in Theorem~\ref{sample_theorem} take into account only the invariant homology generated by a single element of the kernel of the adjacency matrix of $\Gamma$.
By the isomorphism
\begin{displaymath}
H_{1}(M;\mathbb{Z})/\text{torsion} \cong H_{1}(S;\mathbb{Z})^\psi \oplus \mathbb{Z} \cong <x_{0},...,x_{n-1}> \oplus <x_{n}>,
\end{displaymath}
this amounts to setting all except one of the coordinates coming from $\psi$-invariant homology equal to zero. 
In other words, the matrices  $U,V,W$, and~$T$ described in Theorem~\ref{sample_theorem} determine the Teichm\"uller polynomial with all these coordinates specialized to zero. 
\end{proof}
We note that given the algorithm described in this section, it is theoretically possible to describe the full Teichm\"uller polynomial in the setting of Theorem~\ref{sample_theorem}. 
However, compared with the statement of Theorem~\ref{sample_theorem}, 
this would come at the cost of more cumbersome notation. 

\section{Examples}
\label{example_section}

In this section, we perform the computation of the Teichm\"uller polynomial $\Theta_C$ for several examples. 
We first give the result for general alternating-sign $A_n$ diagrams. 
Then, we proceed with a detailed account of the computation for the alternating-sign $A_5$ diagram.
Finally, we give the Teichm\"uller polynomials of two mutant links related by different planar embeddings of their underlying abstract tree $\Gamma$, 
which distinguishes them. 

\subsection{Teichm\"uller polynomial of $A_n$}
We now describe the Teichm\"uller polynomial for $\Gamma = A_{n}$, where $n$ is odd. 
As explained in Section~\ref{constructions-sec}, we consider the fiber surface of the associated alternating-sign Coxeter link. 
It is not hard to check that adding two additional bands along a tree does not change the number of components of the link. 
\begin{figure}[h]
\def\svgwidth{260pt}
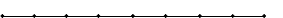
\caption{The $A_9$ graph with the standard bipartite labeling.}
\label{bipartitelabeling}
\end{figure}
Thus, this is a family of two-component links. By picking the standard bipartite vertex labeling of $A_n$, as shown for $n=9$ in Figure~\ref{bipartitelabeling}, 
we can arrange for the adjacency matrix $A$ to have the form
\begin{displaymath}
A=\begin{pmatrix}
0 & X   \\
X^T & 0   
\end{pmatrix},
\end{displaymath}
where $X$ is the matrix of size $(\frac{n+1}{2})\times (\frac{n-1}{2})$ with form (here for $n=9$):

\begin{displaymath}
\begin{pmatrix}
1 & 0 & 0 & 0   \\
1 & 1 & 0 & 0   \\
0 & 1 & 1 & 0   \\
0 & 0 & 1 & 1   \\ 
0 & 0 & 0 & 1   \\

\end{pmatrix}.
\end{displaymath}

Since the links we are considering have two components, the kernel of the adjacency matrix is always generated by a single vector. 
With the choice of the standard bipartite basis of $H_{1}(S;\mathbb{Z})$, we can write
$$
H_{1}(S;\mathbb{Z})^\psi = \left< \sum\limits_{i=0}^{\frac{n+1}{2}} (-1)^i e_{i} \right>.
$$
Notice that the invariant homology is contained in a single set of the bipartition. Thus, we can take $T$, 
the correction matrix for the lifted multitwist defined by the other set of the bipartition, to be the identity matrix $I_{2n-2}$. 
In the notation of the previous section, 

\begin{displaymath}
V=I_{2n-2}+\begin{pmatrix}
0 & 0\\
P & 0   
\end{pmatrix}, \ \ \ 
U=I_{2n-2}+\begin{pmatrix}
0 & Y   \\
0 & 0   
\end{pmatrix},
\end{displaymath}
where $P$ and $Y$ are diagonal block matrices of dimension $n-1$ and of the form
\begin{displaymath}
P=\begin{pmatrix}
J & & \\
 & \ddots & \\ 
 &  & J    
\end{pmatrix}, \ 
 \ \ Y=\begin{pmatrix}
1 &  &  &  &    \\ 
 & R_1 &  &  &    \\
 &  & \ddots &  &    \\ 
 &  &  & R_{\frac{n-3}{2}} &    \\
 &  &  &  & 1  
\end{pmatrix}, 
\end{displaymath}
with blocks 
\begin{displaymath}
J=\begin{pmatrix}
1 & 1\\
1 & 1   
\end{pmatrix}, \ \ \ 
R_{i}=\begin{pmatrix}
1 & 1   \\
x_{0}^{(-1)^{i+1}} & 1  
\end{pmatrix}.
\end{displaymath}
Furthermore, the diagonal correction matrix $W$ can be written as
\begin{displaymath}
W=\begin{pmatrix}
S & 0   \\
0 & N   
\end{pmatrix}, 
\end{displaymath}
where $S$ and $N$ are diagonal matrices 
\begin{displaymath}  S=\begin{pmatrix}
L &  &       \\
 & \ddots &    \\ 
 &  & L    
\end{pmatrix}, \ \ \  N=\begin{pmatrix}
M &  &       \\
 & \ddots &    \\ 
 &  & M    
\end{pmatrix}. 
\end{displaymath}
If $n\equiv3\mod 4$, we have blocks
\begin{displaymath}
M=L=\begin{pmatrix}
x_0 & 0 & 0 & 0   \\
0 & 1 & 0 & 0   \\
0 & 0 & 1 & 0  \\
0 & 0 & 0 & x_0
\end{pmatrix},
\end{displaymath}
and if $n\equiv1\mod 4$, then
\begin{displaymath}
M=\begin{pmatrix}
1 & 0 & 0 & 0   \\
0 & x_0 & 0 & 0   \\
0 & 0 & x_0 & 0  \\
0 & 0 & 0 & 1
\end{pmatrix},
\end{displaymath}
and $\text{dim}(S)=\text{dim}(N)$.
Then we have $P_{E}(x_{0})=UWVT$, $P_{V}(x_{0})$ equals the upper left diagonal block of $W$ of dimension $n-1$, and
\begin{displaymath}
\Theta_{C_{A_n}}(x_{0},x_{1})=\frac{\det(x_{1}I-P_{E}(x_{0}))}{\det(x_{1}I-P_{V}(x_0))}.
\end{displaymath} 

\begin{remark} \emph{
The tree $\Gamma' = A_{2n+1}$ is a vertex extension of $\Gamma$, where $\Gamma$ consists of two disjoint copies of $A_n$. 
Hence, the spectra of the adjacency matrices $A(\Gamma)$ and $A(\Gamma')$ interlace, see, for example,~\cite{BrHa}. 
In particular, each eigenvalue of $A(\Gamma)$ is also an eigenvalue of $A(\Gamma')$, since $A(\Gamma)$ has no simple eigenvalue.
It follows that the single-variable Alexander polynomial of the link associated to $A_n$ divides the single-variable Alexander polynomial of the link associated to $A_{2n+1}$. 
Indeed, these polynomials equal the characteristic polynomial of homological action of the alternating-sign Coxeter mapping class associated to $A_n$ and $A_{2n+1}$, respectively. 
Furthermore, each eigenvalues of the adjacency matrix determines an eigenvalue of the homological action~\cite{HiLi}.  
Interestingly, the same pattern of divisibility seems to hold among Teichm\"uller polynomials associated to the trees $A_n$: 
\begin{align*}
\Theta_{A_{3}}&=
(x_{0}x_{1}-2x_{0}-2x_{1}+1),\\
\Theta_{A_{5}}&=
(x_{0}^2x_{1}^2-4x_{0}^2x_{1}-4x_{0}x_{1}^2+4x_{0}^2+9x_{0}x_{1}+4x_{1}^2-4x_{0}-4x_{1}+1),\\
\Theta_{A_{7}}&=
(x_{0}x_{1}-2x_{0}-2x_{1}+1)(x_{0}^2x_{1}^2-4x_{0}^2x_{1}-4x_{0}x_{1}^2+4x_{0}^2+8x_{0}x_{1}\\
&\ \ \ +4x_{1}^2-4x_{0}-4x_{1}+1),\\
\Theta_{A_{11}}&=
(x_{0}x_{1}-2x_{0}-2x_{1}+1)(x_{0}^2x_{1}^2-4x_{0}^2x_{1}-4x_{0}x_{1}^2+4x_{0}^2+7x_{0}x_{1}\\
&\ \ \ +4x_{1}^2-4x_{0}-4x_{1}+1)(x_{0}^2x_{1}^2-4x_{0}^2x_{1}-4x_{0}x_{1}^2+4x_{0}^2+9x_{0}x_{1}+4x_{1}^2\\
&\ \ \ -4x_{0}-4x_{1}+1).
\end{align*}
These Teichm\"uller polynomials can be calculated using the matrices given above. 
Furthermore, we changed variables so that the polynomials are given in coordinates of the first homology corresponding to the meridians of the link, as described in detail for $A_5$ in Section~\ref{subsectionA_5}.  
}\end{remark}

\subsection{Computation of the Teichm\"uller polynomial for $A_5$}
\label{subsectionA_5}
Let $\Gamma=A_{5}$ and recall the construction of the surface $S$, the monodromy $\psi$ 
and the $\psi$-invariant train track $\tau$ associated to $\Gamma$, as in Section~\ref{constructions-sec}.
The result of this construction is depicted in Figure~\ref{A5_figure}.

\begin{figure}[h]
\def\svgwidth{360pt}
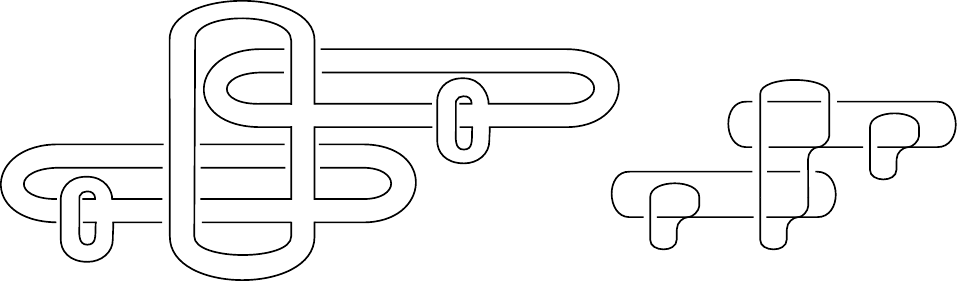
\caption{The abstract surface $S$ and the $\psi$-invariant train track $\tau$ associated to $\Gamma = A_5$.}
\label{A5_figure}
\end{figure}

In order to have a cleaner picture, we consider the graph below as the core of the train track $\tau$. Note that we omitted the smoothings at the vertices and the edges on the ends should be identified with the edge on the opposite side, and that this graph can be naturally embedded into the surface $S$. 
The surface $S$ actually deformation retracts to $\tau$. 
We take the core curves of the Hopf bands $\{a_{1},a_{2},a_{3},b_{1},b_{2}\}$  as a generating set for $H_{1}(S;\mathbb{Z})$ and 
subdivide them at the intersection points to obtain a basis $\{e_0,\dots,e_7\}$ of the edge space of $\tau$, as indicated in Figure~\ref{A5_traintrack}. 

\begin{figure}[h]
\def\svgwidth{250pt}
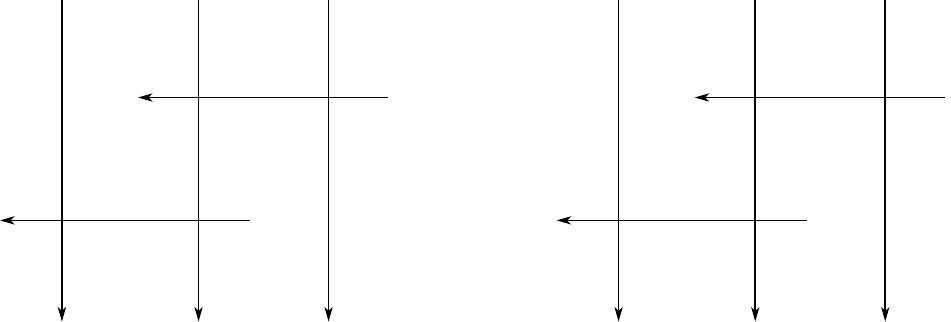
\caption{Homology generators for the surface $S$ and our edge labeling of the train track $\tau$.}
\label{A5_traintrack}
\end{figure}

Recall that the monodromy $\psi:S\to S$ is the product of a positive multitwist along the $b_i$ and a negative multitwist along the $a_j$. 
We also write the map $P_{E} : \tau \rightarrow \tau$ as a product of the actions induced by these two multitwists. 
Let $T_{b}$ be the action on the edge space of $\tau$ induced by the positive multitwist around the loops $\{b_{1},b_{2}\}$. 
Similarly, let $T_{a}$ be the action induced by the negative multitwist around the loops $\{a_{1},a_{2},a_3\}$. 
We adopt the convention that if a curve is to wind around another, it does so to the right or below the curve. 
Then, using the edge labels as in Figure~\ref{A5_traintrack}, we arrive at the following matrices for $T_{a}$ and $T_{b}$. 

\begin{displaymath}
T_{b} =\begin{pmatrix}
1 & 0 & 0 & 0 & 0 & 0 & 0 & 0 \\
0 & 1 & 0 & 0 & 0 & 0 & 0 & 0 \\
0 & 0 & 1 & 0 & 0 & 0 & 0 & 0 \\
0 & 0 & 0 & 1 & 0 & 0 & 0 & 0 \\ 
1 & 1 & 0 & 0 & 1 & 0 & 0 & 0 \\
1 & 1 & 0 & 0 & 0 & 1 & 0 & 0 \\
0 & 0 & 1 & 1 & 0 & 0 & 1 & 0 \\ 
0 & 0 & 1 & 1 & 0 & 0 & 0 & 1
\end{pmatrix},\ 
T_{a}= \begin{pmatrix}
1 & 0 & 0 & 0 & 1 & 0 & 0 & 0 \\
0 & 1 & 0 & 0 & 0 & 1 & 1 & 0 \\
0 & 0 & 1 & 0 & 0 & 1 & 1 & 0 \\
0 & 0 & 0 & 1 & 0 & 0 & 0 & 1 \\ 
0 & 0 & 0 & 0 & 1 & 0 & 0 & 0 \\
0 & 0 & 0 & 0 & 0 & 1 & 0 & 0 \\
0 & 0 & 0 & 0 & 0 & 0 & 1 & 0 \\ 
0 & 0 & 0 & 0 & 0 & 0 & 0 & 1
\end{pmatrix}.
\end{displaymath}

In order to compute the Teichm\"uller polynomial, we need to find the induced map on the edges of the train track lifted to the cover defined by $\pi_{1}(S)\rightarrow H_{1}(S;\mathbb{Z})^\psi$. 
Since the surface in this example has two boundary components, we have $H_{1}(S;\mathbb{Z})^{\psi} \cong \mathbb{Z}$. 
The cover is then constructed by cutting the surface along the cohomology classes that are invariant under the monodromy $\psi$ and attaching them to form a $\mathbb{Z}$-covering. 
It is verified directly that the invariant homology for this example is generated by $a_{1}-a_{2}+a_{3}$, 
for example, by calculating the kernel of the adjacency matrix for $A_5$. Thus, winding around $a_{1}$ or $a_{3}$ leads one layer higher in the cover. Winding around $a_{3}$ leads one layer lower. 
The covering surface $\widetilde S$ is visualized in Figure~\ref{A5_cover}. 

\begin{figure}
\def\svgwidth{210pt}
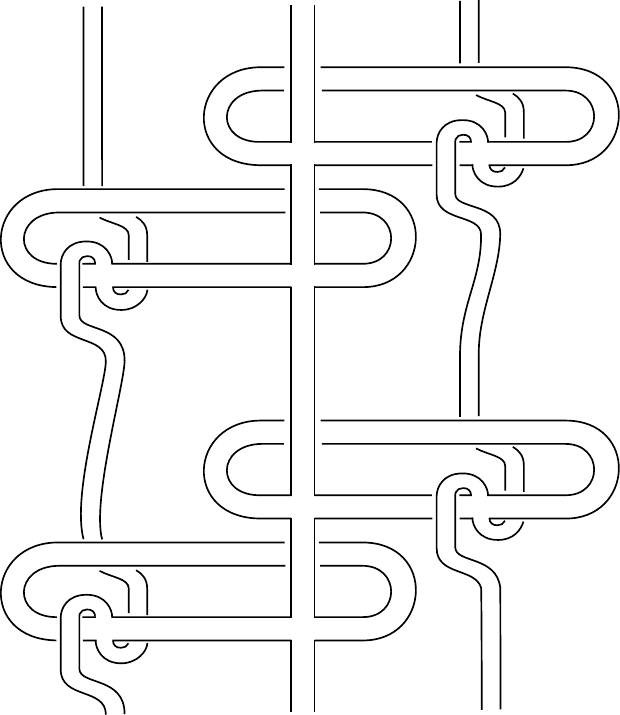
\caption{Two copies of a fundamental domain of the covering surface $\widetilde S$ corresponding to $H_{1}(S;\mathbb{Z})^{\psi}$.}
\label{A5_cover}
\end{figure}

We are interested in the induced action of $\widetilde\psi$ on the edge space of the lifted train track $\widetilde\tau$, regarded as a $\Z[H]$-module.
We factor this action as $P_{E} = \widetilde{T}_{a}\widetilde{T}_{b}$ where $\widetilde{T}_{a} = UW$ and $\widetilde{T}_{b} = V$, as in Section~\ref{algorithm_section}: 
the matrices $U$ and $V$ are obtained by analyzing a single fundamental domain in the cover $\widetilde S$ 
and $W$ is a diagonal matrix that corrects some columns to their proper levels. 
Note that no correction matrix $T$ is needed to correct $V$ to $\widetilde{T}_{b}$, since the $\psi$-invariant homology $H_1(S;\mathbb{Z})^\psi$ is supported entirely in the subspace of $H_1(S;\mathbb{Z})$ generated by the curves $a_i$. 
A fundamental domain of $\widetilde\tau$ well suited to our purposes is drawn in Figure~\ref{fundamentalfigure}.
\begin{figure}
\def\svgwidth{200pt}
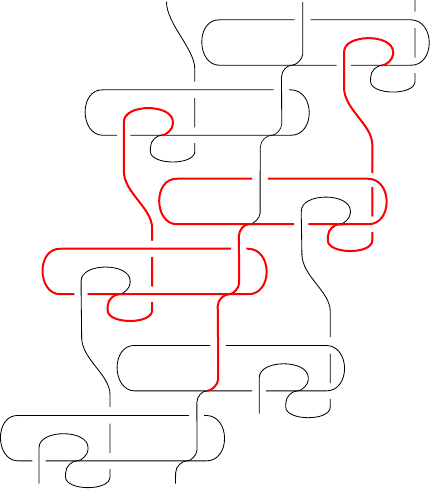
\caption{The lifted train track $\widetilde\tau$ and the fundamental domain we use in our calculations, drawn in thick red.}
\label{fundamentalfigure}
\end{figure}
We caluclate $U$ and $V$ as

\begin{displaymath}
 V =\begin{pmatrix}
1 & 0 & 0 & 0 & 0 & 0 & 0 & 0 \\
0 & 1 & 0 & 0 & 0 & 0 & 0 & 0 \\
0 & 0 & 1 & 0 & 0 & 0 & 0 & 0 \\
0 & 0 & 0 & 1 & 0 & 0 & 0 & 0 \\ 
1 & 1 & 0 & 0 & 1 & 0 & 0 & 0 \\
1 & 1 & 0 & 0 & 0 & 1 & 0 & 0 \\
0 & 0 & 1 & 1 & 0 & 0 & 1 & 0 \\ 
0 & 0 & 1 & 1 & 0 & 0 & 0 & 1
\end{pmatrix},\ 
 U= \begin{pmatrix}
1 & 0 & 0 & 0 & 1 & 0 & 0 & 0 \\
0 & 1 & 0 & 0 & 0 & 1 & 1 & 0 \\
0 & 0 & 1 & 0 & 0 & x_{0}^{-1} & 1 & 0 \\
0 & 0 & 0 & 1 & 0 & 0 & 0 & 1 \\ 
0 & 0 & 0 & 0 & 1 & 0 & 0 & 0 \\
0 & 0 & 0 & 0 & 0 & 1 & 0 & 0 \\
0 & 0 & 0 & 0 & 0 & 0 & 1 & 0 \\ 
0 & 0 & 0 & 0 & 0 & 0 & 0 & 1
\end{pmatrix}.
\end{displaymath}
The map on the train track $\tau$ induced by the negative multitwist along the $a_i$ sends the edge $e_5$ to the edge path $e_5e_1e_2$. 
Tracing the lift of this edge path starting from the corresponding edge $e_5$ in the chosen fundamental domain, 
we see that we leave the fundamental domain and pass in the downward ($x_0^{-1}$) direction when tracing the lift of the edge $e_2$. 
This accounts for the entry $x_0^{-1}$ in the matrix $U$.

The last thing to account for is the fact that when lifting the negative multitwist along the $a_i$ to the train track $\widetilde\tau$, 
the starting point of the image of an edge from the fundamental domain might not lie in the fundamental domain. 
We decide to use the intersection of the curves $b_1$ and $a_2$ as basepoint for the lifts. 
In this case, the lift of the negative multitwist along the $a_i$ sends the edges $e_0$, $e_3$, $e_5$, and $e_6$ 
to edge paths starting one level shifted by $x_0$ in the lifted train track $\widetilde\tau$.
The images of the other edges start in the fundamental domain, and so does the image of every edge under the lifted positive multitwist along the $\beta_j$. 
This gives the matrix $W$, with which we have to correct $U$ in order to obtain $\widetilde T_a$:

\begin{displaymath}
W =\begin{pmatrix}
x_{0} & 0 & 0 & 0 & 0 & 0 & 0 & 0 \\
0 & 1 & 0 & 0 & 0 & 0 & 0 & 0 \\
0 & 0 & 1 & 0 & 0 & 0 & 0 & 0 \\
0 & 0 & 0 & x_{0} & 0 & 0 & 0 & 0 \\ 
0 & 0 & 0 & 0 & 1 & 0 & 0 & 0 \\
0 & 0 & 0 & 0 & 0 & x_{0} & 0 & 0 \\
0 & 0 & 0 & 0 & 0 & 0 & x_{0} & 0 \\ 
0 & 0 & 0 & 0 & 0 & 0 & 0 & 1
\end{pmatrix}.
\end{displaymath}
The determinant formula stated in Section~\ref{background_section} also requires the action of the lifted monodromy on the vertices $\widetilde\tau$. 
Since the action on vertices is the same as on starting points of edges, 
this action is given by
\begin{displaymath}
P_{V} =\begin{pmatrix}
x_{0} & 0 & 0 & 0  \\
0 & 1 & 0 & 0  \\
0 & 0 & 1 & 0 \\
0 & 0 & 0 & x_{0} \\ 
\end{pmatrix},
\end{displaymath}
the upper diagonal block of $W$.
All of our calculations so far have been in coordinates $H_{1}(M)\cong H_{1}(S)^{\psi} \oplus \mathbb{Z} \cong \langle x_{0},x_1\rangle$. 
The usual basis for the homology of link exteriors is the set of meridians of the link. If we push $a_{1},-a_{2}$, or $a_{3}$ off the embedded fiber surface $S$ as shown in Figure~\ref{basechange}, we can find the isomorphism needed the move to the usual coordinates.
\begin{figure}
\def\svgwidth{260pt}
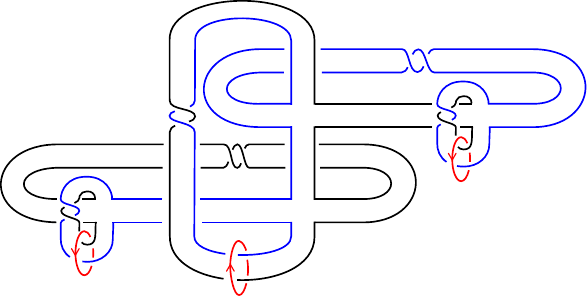
\caption{The pushoffs $a_1^\sharp$, $-a_2^\sharp$ and $a_3^\sharp$ of $a_1$, $-a_2$ and $a_3$, respectively.}
\label{basechange}
\end{figure}
We can see that the isomorphism is 
\begin{align*}
x_{0} &\mapsto y_{0}y_{1}^{-1}, \\
x_{1} &\mapsto y_{1}.
\end{align*}
If we compute 
\begin{displaymath}
\Theta_{A_5}(x_{0},x_{1})=\frac{\det(x_{1}I-P_{E}(x_{0}))}{\det(x_{1}I-P_{V}(x_0))}
\end{displaymath}
and then change variables as described, we obtain
\begin{displaymath}
\Theta_{A_5}(y_{0},y_{1})=(y_{0}^2y_{1}^2-4y_{0}^2y_{1}-4y_{0}y_{1}^2+4y_{0}^2+9y_{0}y_{1}+4y_{1}^2-4y_{0}-4y_{1}+1).
\end{displaymath}
Notice this is also (up to units) the Alexander polynomial for this link. 
The polynomial is associated to the face of the Thurston norm ball whose cone contains the homology class of $S$. 
This face is described by the hyperplane $2x_{0}+2x_{1}=1$ in the first quadrant of $H_{2}(M,\partial M;\mathbb{R})\cong \mathbb{R}^2$.

\subsection{Mutant links}
The links obtained from distinct planar embeddings of an abstract tree $\Gamma$ might be different.
For example, in the classical case where all the signs are positive, this is true for a large class of trees by a theorem of Gerber~\cite{Ge}.
In any case, different embeddings can not be distinguished by the single variable Alexander or Teichm\"uller polynomial. 
Indeed, these polynomials only depend on the abstract tree $\Gamma$ and not the chosen embedding. 
While the train tracks corresponding to two distinct embeddings have the same incidence matrix, and the invariant homology of the monodromy maps are defined in terms of $A(\Gamma)$, 
the distribution of the invariant homology throughout the varying train tracks can be very different. 
With this in mind, it is not surprising that in some cases the multivariable Teichm\"uller polynomial distinguishes such links.  
\begin{figure}[h]
\def\svgwidth{200pt}
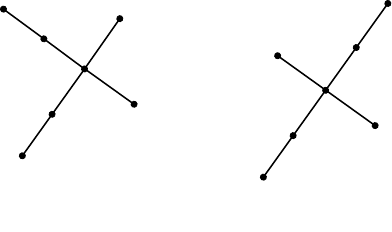
\caption{Two different plane trees $\Gamma_1$ and $\Gamma_2$ with the same underlying abstract tree.}
\label{mutant_trees}
\end{figure}
\begin{example}\emph{
Consider the two embeddings $\Gamma_1$ and $\Gamma_2$ of the tree $\Gamma$ shown in Figure~\ref{mutant_trees}. 
By Remark~\ref{mutation_remark}, the two alternating-sign Coxeter links associated to $\Gamma_1$ and $\Gamma_2$ are mutant. 
We now give their Teichm\"uller polynomials and see that they are distinct.
We consider the ambient manifolds of the alternating Coxeter links defined by these plane trees.
In a basis of the first homology consisting of meridians of the corresponding links, 
the Teichm\"uller polynomials $\Theta_1$ and $\Theta_2$ of the fibered cones associated to $\Gamma_1$ and $\Gamma_2$ are given by
\begin{align*}
\Theta_1 =&\ (x_{1}^2-3x_{1}+1)(x_{0}x_{1}^3-7x_{0}x_{1}^2-2x_{1}^3+9x_{0}x_{1}+9x_{1}^2-2x_{0}+7x_{1}+1), \\
\Theta_2  =&\ x_0^3x_1^3 - 6x_0^3x_1^2 - 6x_0^2x_1^3 + 8x_0^3x_1 + 30x_0^2x_1^2 + 8x_0x_1^3 - 2x_0^3 - 34x_0^2x_1 \\
                & - 34x_0x_1^2 - 2x_1^3 + 8x_0^2 + 30x_0x_1 + 8x_1^2 - 6x_0 - 6x_1 + 1,
\end{align*}
which are distinct. The Teichm\"uller polynomials $\Theta_1$ and $\Theta_2$ can be calculated, 
for example, by using Theorem~\ref{sample_theorem} and changing to the bases consisting of meridians.
Note that since both links have two components, Theorem~\ref{sample_theorem} can be used to compute the whole Teichm\"uller polynomial.
The change of base is calculated explicitly. However, it does not change the fact that $\Theta_1$ factors as a product of two polynomials while $\Theta_2$ does not factor.
In particular, the change of base is not needed to show that $\Theta_1$ and $\Theta_2$ are distinct.}
\end{example}

\end{document}